\documentclass[11pt]{amsart}
\usepackage{amsmath}
\usepackage{amsfonts,amssymb}
\usepackage{euscript}
\usepackage{amsthm}
\usepackage[matrix,arrow,curve]{xy}
\numberwithin{equation}{section}
\theoremstyle{plain}
\newtheorem{theorem}{Theorem}[section]
\newtheorem{lemma}[theorem]{Lemma}
\newtheorem{predl}[theorem]{Proposition}
\newtheorem{corollary}[theorem]{Corollary}

\newtheorem{conjecture}[theorem]{Conjecture}
\theoremstyle{definition}
\newtheorem{definition}[theorem]{Definition}
\newtheorem{remark}[theorem]{Remark}
\newtheorem{example}[theorem]{Example}
\newcommand{\R}{\mathbb R}

\newcommand{\Z}{\mathbb Z}

\renewcommand{\P}{\mathbb P}

\newcommand{\D}{\mathcal D}

\newcommand{\EE}{\mathcal E}

\renewcommand{\O}{\mathcal O}

\renewcommand{\k}{\mathsf k}

\newcommand{\coh}{\mathrm{coh}}

\newcommand{\xra}{\xrightarrow}
\renewcommand{\le}{\leqslant}
\renewcommand{\ge}{\geqslant}

\DeclareMathOperator{\Hom}{\textup{Hom}}

\DeclareMathOperator{\Pic}{\mathrm{Pic}}

\DeclareMathOperator{\rank}{\mathrm{rank}}

\usepackage{array}
\usepackage{tikz}
\usetikzlibrary{arrows,automata}
\usetikzlibrary{topaths}
\begin{document}

\title{Note on Applications of Toric Systems on surfaces}

\author{shizhuo ZHANG}
\thanks{Indiana University}
\address{Department of Mathematics, Indiana University, 831 E. Third St., Bloomington, IN 47405, USA}
\email{zhang398@umail.iu.edu}
\maketitle

\begin{abstract}
  In this note we apply the techniques of the toric systems introduced by Hille-Perling to several problems on smooth projective surfaces: We showed that the existence of full exceptional collection of line bundles implies the rationality for small Picard rank surfaces; we proved equivalences of several notions of cyclic strong exceptional collection of line bundles; we also proposed a partial solution to a conjecture on exceptional sheaves on weak del Pezzo surfaces.
\end{abstract}  

\tableofcontents
  
\section{Introduction}
In this work, we study exceptional collection of line bundles on surfaces. Let $X$ be a smooth projective surface over an algebraically closed field $k$ of characteristic $0$. An object $\mathcal{E}$ in the derived category $D^b(X)$ is called exceptional if 
$$Hom_{D^b(X)}(\mathcal{E},\mathcal{E}[i])=
\begin{cases}
k\quad\text{for}\quad i=0,\\
0\quad\text{for}\quad i\neq 0.
\end{cases}
$$

A collection of line bundles $(E_1,E_2,\ldots, E_n)$ in the derived category $D^b(X)$ form a numerically exceptional collection of line bundles of maximal length if:\begin{enumerate}
    \item Each line bundle $E_i$ is numerically exceptional: $\chi(E_i,E_i)=1$
    \item $\chi(E_i,E_j)=0$ for all $i>j$.
    
The numerically exceptional collection of line bundles is of maximal length if 
\item $n=rank(K_0^{num}(X))=rank(N^1(X))+2$, where $K_0^{num}(X)$ is the numerical Grothendieck group and $N^1(X)$ is numerical Picard group of $X$. 
\end{enumerate}

An exceptional collection of line bundles on $X$ is a numerically exceptional collection of line bundles $(E_1,\ldots,E_n)$ such that \begin{enumerate}
    \item Each line bundle $E_i$ is exceptional: $Ext^m(E_i,E_i)=k$ if $m=0$ and zero otherwise.
    \item $Ext^m(E_i,E_j)=0$ for all $i>j$.
    
    The collection is called strong if in addition:
    \item $Ext^m(E_i,E_j)=0$ for $m\neq 0$ and all $i,j$.
\end{enumerate}

An exceptional collection of line bundles $(E_1,\ldots,E_n)$ is called full if $D^b(X)$ is the smallest full strict triangulated subcategory of $D^b(X)$ which contains $E_1,\ldots,E_n$.

The paper \cite{HP} by Hille and Perling contains the first systematic study of full exceptional collection of line bundles on smooth rational surfaces. In particular, they introduced a technical notion called toric system as an equivalent notion of (numerically) exceptional collection of line bundles. It is defined as follows. Let $(O_X(D_1),\ldots,O_X(D_n))$ be a numerically exceptional collection of line bundles on a smooth projective surface with $\chi(O_X)=1$. Consider differences between the $D_i$ and put:

$$A_i=
\begin{cases}
D_{i+1}-D_i\quad\text{for}\quad 1\le i\le n-1,\\
(D_1-K_X)-D_n=-K_X-(A_1+\ldots+A_{n-1})\quad\text{for}\quad i=n.
\end{cases}
$$

It was shown in \cite{HP} that the system of divisors $(A_1,\ldots,A_n)$ satisfies the following properties:\begin{enumerate}
\item $A_i\cdot A_{i+1}=A_n\cdot A_1=1$ for $i=1,\ldots,n-1$;
\item $A_i\cdot A_j=0$ for $i+1<j$ unless $i=1,j=n$;
\item $A_1+\ldots+A_n=-K_X$.
\end{enumerate}

A system of divisors $(A_1,\ldots,A_n)$ satisfying the three properties above is called the toric system associated with the numerically exceptional collection of line bundles $(O(D_1),\ldots, O(D_n))$ on $X$. Among other theorems, L.Hille and M.Perling proved the following remarkable statement, which is starting point of our work:
\begin{theorem}\cite[Proposition 2.7]{HP} or \cite[Theorem 3.5]{V}
Let $X$ be a smooth projective surface with $\chi(O_X)=1$ with a toric system $(A_1,\ldots,A_n)$ of maximal length. Then there is a smooth projective toric surface $Y$ with torus-invariant irreducible divisors $D_1,\ldots,D_n$ (listed in the cyclic order), such that $D_i^2=A_i^2$ for all $i$. 
\end{theorem}

Having a (full, maximal length) exceptional collection of line bundles in derived category of coherent sheaves on surfaces is a nice but rare property. It is interesting to understand the obstruction of the existence of such collections on a smooth projective surface. It is known that every smooth rational surface has a full exceptional collection of line bundles by \cite{HP} or Orlov's blow-up formula for derived categories. But it is still an open question whether rational surface is the only one having such collections. The Orlov's folklore conjecture states that a smooth projective surface admitting a full exceptional collection of line bundles must be a rational surface. In \cite{V}, C.Vial classified all complex surfaces with $p_g=q=0$ with numerically exceptional collection of line bundles of maximal length. M.Brown and I.Shipman proved Orlov's folklore conjecture in the case of full strong exceptional collection of line bundles, see~\cite[Theorem 4.4]{BS}. In \cite{EXZ}, we showed that every smooth projective surface admitting a cyclic strong exceptional collection of line bundles (See Definition $1.5$) of maximal length must be a rational surface, moreover, a weak del Pezzo surface. The main results of this article are to prove more statements in this direction:
\begin{theorem}
Let $X$ be a smooth projective surface admitting a strong exceptional collection of line bundles of maximal length. Then $X$ must be a rational surface.
\end{theorem}

An immediate corollary of this theorem is that a surface admitting a cyclic strong exceptional collection of line bundles of maximal length must be a rational surface. An alternating proof of this fact is given in \cite[Lemma 15.1]{EXZ}

Motivated by the desire to describe geometric phantom categories: subcategories of $D^b(X)$ with trivial Grothendieck group $K_0$ and Hochschild homology $HH_0$, a large amount of work was carried out to establish the exceptional collection of line bundles of maximal length on a surface of general type with $p_g=q=0$, see~\cite{AO},\cite{BGS},\cite{GS},\cite{GO}. It immediately follows from Theorem $1.2$ that:

\begin{corollary}
Let $X$ be a surface of general type with $p_g=q=0$. Then a strong exceptional collection of line bundles of maximal length does not exist.
\end{corollary}

Although it seems that Orlov's folklore conjecture is out of reach at present, we are able to prove the conjecture for small Picard rank surfaces:

\begin{theorem}
Let $X$ be a smooth projective surface with $rk N^1(X)\leq 3$ admitting a full exceptional collection of line bundles. Then $X$ must be a rational surface.
\end{theorem}

This means that if the length of a maximal length exceptional collection of line bundles on surface of general type with $p_g=q=0$ is $3, 4$ or $5$, it cannot be full. 

Another important aspect of exceptional collections of line bundles are cyclic strong exceptional collections of line bundles. We recall its definition as follows:
\begin{definition}
Let $(E_1,\ldots,E_n)$ be a strong exceptional collection of line bundles on a smooth projective surface $X$. It is called cyclic strong if $$Ext^k(E_j,E_i\otimes\omega_X^{-1})=0$$ for $k\geq 1$ and $j>i$. In this case, a tilting bundle(See Definition $8.1$) formed by such collection of line bundles is called cyclic strong.
\end{definition}

Cylic strong exceptional collections of line bundles are investigated in many areas of mathematics. There are mainly three other versions of such collection under different names. We recall them as follows:

\begin{definition}[\cite{B}, \cite{BF}]
Let $X$ be a smooth projective surface and $\pi: Tot(\omega_X)\rightarrow X$ be the natural projection from the total space of the canonical bundle. We say that a tilting bundle, $T$, is pull back if $\pi^*(T)$ is tilting. If $T$ is the sum of an exceptional collection of line bundles, we say the corresponding collection of line bundles is a pull back exceptional collection of line bundles.
\end{definition}

\begin{definition}[\cite{CHAN}]
Let $X$ be a smooth projective surface. A tilting bundle $T$ on $X$ is $2$-hereditary if $Ext^k(T,T\otimes\omega_X^{-i})=0$ for all $i\geq 0$ and $k\geq 1$.
If in addition, $T$ is a direct sum of exceptional collection of line bundles, then such collection is called a $2$-hereditary exceptional collection of line bundles
\end{definition}

\begin{definition}[\cite{BF}]
Let $X$ be a smooth projective surface and $T$ is a tilting bundle which is a direct sum of a strong exceptional collection of line bundles:$(E_1,\ldots,E_n)$. Such collection is called almost pull back if the following condition holds:
$$Ext^k(E_i,E_j\otimes\omega_X^{-1})=0$$ for all $i,j$ and $k\geq 1$. In this case, $T$ is called an almost pull back tilting bundle.
\end{definition}

In \cite[Remark 3.25]{BF}, Ballard and Favero asked the question on the precise relations of those four notions. We are able to answer this question by showing the following statement.

\begin{theorem}
Let $\mathbb{E}:=(E_0,\ldots,E_n)$ be an exceptional collection of line bundles on a smooth projective surface $X$. Then $\mathbb{E}$ is cyclic strong iff it is pull back iff it is $2$-hereditary iff it is almost pull back.
\end{theorem}

The counterpart of exceptional vector bundles in derived category of coherent sheaves $D^b(X)$ on a surface $X$ is an exceptional sheaf with torsion subsheaf. Sometimes, it is more convienient to study exceptional sheaf with torsion via corresponding exceptional vector bundles. The exceptional objects on a del Pezzo surface are well understood thanks to the work of Kuleshov and Orlov. In \cite{KO}, they showed that any exceptional object on a del Pezzo surface is isomorphic to a shift of an exceptional vector bundle or a line bundle on a $(-1)$-curve. Thus, it is natural to consider the classification of exceptional objects on a weak del Pezzo surface  (The anti-canonical divisor $-K_X$ is big and nef). In this case, the exceptional objects are not as simple as those on del Pezzo surfaces due to the existence of $(-2)$-curves on weak del Pezzo surfaces. Instead, one can expect to classify the exceptional sheaves on $X$. In \cite{KU}, S.Kuleshov proved the following theorem:
\begin{theorem}
Suppose $\mathcal{E}$ is an exceptional sheaf on a weak del Pezzo surface $X$. We have the following cases:\begin{enumerate}
    \item $\mathcal{E}$ is locally free
    \item $\mathcal{E}$ has torsion subsheaf supported on $(-2)$-curves.
    \item $\mathcal{E}\cong O_C(d)$ for some rational curve $C$ and $C^2=-1$.
    \item $\mathcal{E}$ is a torsion exceptional sheaf whose support contains exactly one $(-1)$-curve and some $(-2)$-curves.
\end{enumerate}
\end{theorem}

In the \cite{OU}, Okawa and Uehara proposed the following conjecture:
\begin{conjecture}
Let $X$ be a weak del Pezzo surface. For any exceptional sheaf $\mathcal{E}\in D(X)$, there exists an auto-equivalence $\Phi\in Auteq(D(X))$ such that $\Phi(\mathcal{E})$ is an exceptional vector bundle, or a line bundle on a $(-1)$-curve on $X$.
\end{conjecture}

In paper \cite{CJ}, P.Cao and C.Jiang proved this conjecture in the case of torsion exceptional sheaves on weak del pezzo surface of type $A$ with $K_X^2\geq 3$  (i.e. those $X$ with at most $A_n$-singularities on its canonical model). Precisely, they proved the following theorem:

\begin{theorem}
Let $X$ be a weak del Pezzo surface of type $A$ with $K_X^2\geq 3$ and $\mathcal{E}\in D^b(X)$ a torsion exceptional sheaf on $X$. Then there exists a $(-1)$-curve $D$, an integer $d$ and a sequence of auto-equivalences (called spherical twist functors): $\Phi_1,\ldots,\Phi_n$ associated to line bundles on chain of $(-2)$-curves such that:
$$\mathcal{E}\cong\Phi_1\circ\Phi_2\circ\ldots\circ\Phi_n(O_D(d))$$
\end{theorem}

We are able to prove the conjecture in a more general situtation if torsion exceptional sheaves $\mathcal{E}\in D^b(X)$ are of the form $O_D$, where $D$ is a codimension $1$ subscheme, not necessarily reduced.

\begin{theorem}
Let $X$ be a weak del Pezzo surface such that $K_X^2\geq 2$, and let $\mathcal{E}$ be a torsion exceptional sheaf of the form $O_D$, where $D$ is subscheme of codimension $1$, not neccesarily reduced. Then there exists a series of spherical twist functors associated to line bundles on $(-2)$-curves: $T_i\in Auteq(D(X))$ and an $(-1)$-curve $D$ on $X$, such that $T_n\circ T_{n-1}\circ\ldots\circ T_1 (O_D)=O_E$.
\end{theorem}

The outline of the note is organized as follows: we devote section $2$ to the basic definitions and properties of (numerically, strong) left-orthogonal divisors on surfaces. In section $3$, we give an overview of definitions of weak del Pezzo surfaces and complete classification of (numerically, strong) left-orthogonal divisors on weak del Pezzo surfaces. In section $4$, we give a review of definitions and important properties of toric systems on weak del Pezzo surfaces. In section $5$, we review admssible sequences, whose classification will be the key to prove Theorem $1.2$ and $1.4$. Most part of those sections can be found in \cite{HP} and \cite{EXZ}. 

In section $6$, we overview the definition of anticanonical pseudo height of an exceptional collections and related properties. In section $7$ we prove Theorem $1.2$ and $1.4$. In section $8$, we prove Theorem $1.8$ and give a classification of surfaces which admit $2$-hereditary tilting bundles given by a direct sum of exceptional collections of line bundles. In section $9$, we prove Theorem $1.12$. 

\medskip
{\bf Acknowledgements.} I would like to thank Alexey Elagin, Junyan Xu and Xuqiang Qin for useful discussions, Dylan Spence for valuable comments. I am very grateful to my advisor, Valery Lunts, for his support.

\section{Divisors on surfaces and $r$-classes}

Let $X$ be a smooth projective surface with $p_g=q=0$ over an algebraically closed field $\k$ of characteristic zero.  Let $K_X$ be a canonical divisor on $X$.
%Let $X$ be a smooth projective rational surface over $\C$ having $\P_2$ as a minimal %model. 
Let $d=K_X^2$ be the degree of $X$.
The numerical Picard group $N^1(X)$ of $X$ is a finitely generated abelian group of rank $10-d$.
It is equipped with the intersection form $(D_1,D_2)\mapsto D_1\cdot D_2$. If $X$ is a rational surface, then $rk Pic(X)=10-d$ and the intersection form has signature $(1,9-d)$.
%Then the group $\Pic X$ has the basis $L,E_1,\ldots,E_n$ where $n=9-d$. Intersection %form in this basis is given by $L^2=1, E_i^2=-1, L\cdot E_i=0, E_i\cdot E_j=0$ for $i\ne %j$.
For a divisor $D$ on $X$, we will use the following shorthand notations:
$$H^i(D):=H^i(X,\O_X(D)),\quad h^i(D)=\dim H^i(D),\quad \chi(D)=h^0(D)-h^1(D)+h^2(D).$$
By the Riemann-Roch formula, one has
$$\chi(D)=1+\frac{D\cdot (D-K_X)}2.$$
 
The following notions are introduced in \cite[Definition 3.1]{HP}. 
\begin{definition}
A divisor $D$ on $X$ is \emph{numerically left-orthogonal} if $\chi(-D)=0$ (or equivalently $D^2+D\cdot K_X=-2$).
A divisor $D$ on $X$ is \emph{left-orthogonal} (or briefly \emph{lo}) if $h^i(-D)=0$ for all $i$. A divisor $D$ on $X$ is \emph{strong left-orthogonal} (or briefly \emph{slo}) if $h^i(-D)=0$ for all $i$ and $h^i(D)=0$ for $i\ne 0$.
%We call an element $D\in\Pic X$ \emph{irreducible} if $D$ corresponds to an irreducible %divisor, and \emph{reducible} otherwise.
\end{definition} 

\begin{definition}
We call $D$ an \emph{$r$-class} if $D$ is numerically left-orthogonal and $D^2=r$. \end{definition}

Motivation: if $C\subset X$ is a smooth rational irreducible curve, then the class of $C$ in $\Pic X$ is an $r$-class where $r=C^2$. 

If $C$ is an irreducible reduced curve on $X$ and $r=C^2$, it is said that $C$ is an \emph{$r$-curve}. An $r$-curve is \emph{negative} if $r<0$.

The next propositions are easy consequences of the Riemann-Roch formula, see~\cite[Lemma 3.3, ]{HP} or \cite[Lemma 2.10, Lemma 2.11]{EL}.
\begin{predl}
Let $D$ be a numerically left-orthogonal divisor on $X$. Then 
$$\chi(D)=D^2+2=-D\cdot K_X.$$
\end{predl}
\begin{predl}
\label{prop_classes0}
Suppose $D_1,D_2$ are numerically left-orthogonal divisors on $X$. 
Then $D_1+D_2$ is numerically left-orthogonal if and only if $D_1D_2=1$. If that is the case, then 
$$\chi(D_1+D_2)=\chi(D_1)+\chi(D_2)\qquad \text{and}\qquad (D_1+D_2)^2=D_1^2+D_2^2+2.$$
\end{predl}

Let $X$ be a rational surface of degree $d=K_X^2\geq 1$. Denote the set of $(-2)$-classes on $X$ by $R(X)$. It is a root system in some subspace in $N_X=(K_X)^{\perp}\subset \Pic(X)\otimes\R$ (see Manin's book \cite{Ma}) and depends only on $\deg(X)$. If $\deg(X)\le 6$ then $R(X)$ spans $N_X$. 

\begin{table}[h]
\caption{Root systems $R(X)$}
\begin{center}
\begin{tabular}{|c|c|c|c|c|c|c|c|}
		 \hline
			degree & 7 & 6& 5&4&3&2&1\\
		 \hline
			type & $A_1$ & $A_1+A_2$ & $A_4$ & $D_5$ & $E_6$ & $E_7$ & $E_8$\\
		 \hline	
			$|R(X)|$ & $2$ & $8$ & $20$ & $40$ & $72$ & $126$ & $240$\\
		 \hline	
			$|I(X)|$ & $3$ & $6$ & $10$ & $16$ & $27$ & $56$ & $240$\\
		 \hline	
\end{tabular}
\end{center}

\label{table_root}
\end{table}

Denote by $R^{eff}(X)\subset R(X)$ the subset of effective $(-2)$-classes. Denote by $R^{irr}(X)\subset R^{eff}(X)$ the subset of irreducible $(-2)$-curves. 
Denote by $R^{slo}(X)\subset R^{lo}(X)\subset R(X)$ the subsets of strong left-orhogonal and left-orthogonal divisors respectively. Again, $R^{eff}(X),R^{irr}(X),R^{lo}(X)$ and $R^{slo}(X)$ depend on the surface $X$. 

\section{Weak del Pezzo surfaces}
\label{section_wdp}
By definition, a weak del Pezzo surface is a smooth connected projective rational surface~$X$ such that $K_X^2>0$ and $-K_X$ is nef. A del Pezzo surface is a  smooth connected projective rational surface $X$ such that $-K_X$ is ample. We refer to Dolgachev~\cite[Chapter 8]{Do} or Derenthal~\cite{DE} for the main properties of weak del Pezzo surfaces. Every weak del Pezzo surface $X$ except for Hirzebruch surfaces $\mathbb{F}_0$ and $\mathbb{F}_2$ is a blow-up of $\P^2$ at several (maybe infinitesimal) points. That is, there exists a sequence 
$$X=X_n\xra{p_n} X_{n-1}\xra\ldots X_1\xra{p_1} X_0=\P^2$$
of $n$ blow-ups, where $p_k$ is the blow-up of point $P_k\in X_{k-1}$. 
Moreover, the surface $X_n$ as above is a weak del Pezzo surface if and only if $n\le 8$ and for any $k$ the point $P_k$ does not belong to a $(-2)$-curve on $X_{k-1}$.

Let $X$ be a blow-up of $\P^2$ at $n$ (maybe infinitesimal) points. The Picard group $\Pic X$ of $X$ has the standard basis $L,E_1,\ldots,E_n$. The  intersection form is given by 
$$L^2=1,E_i^2=-1,L\cdot E_i=0, E_i\cdot E_j=0 \quad \text{for}\quad i\ne j.$$
We will use the following shorthand notation:
$$E_{i_1\ldots i_k}=E_{i_1}+\ldots+E_{i_k},\quad L_{i_1\ldots i_k}=L-E_{i_1\ldots i_k}.$$ The next lemma is well known, check \cite[Lemma 3.1]{EXZ} for details.

\begin{lemma}\label{neg curve}
Let $C\subset X$ be a negative curve on a weak del Pezzo surface $X$.  Then $C$ is a smooth rational curve and $C^2=-2$ or $-1$. If $X$ is a del Pezzo surface then $C^2=-1$.
\end{lemma} 
%\begin{proof}
%By the Riemann-Roch formula, one has 
%$$\chi(\O_C)=\chi(\O_X)-\chi(\O_X(-C))=-\frac{C^2+CK_X}{2}.$$
%Since $C$ is irreducible and reduced, one has $h^0(\O_C)=1$. Hence
%$$0\le h^1(\O_C)=1+\frac{C^2+CK_X}2.$$
%Note that $CK_X\le 0$ because $-K_X$ is nef and $C^2<0$ by assumptions. It follows that $h^1(\O_C)=0$, hence $C$ is rational and smooth. Moreover, $C^2+CK_X=-2$, therefore $C^2=-2$ or $-1$. If $X$ is a del Pezzo surface, then $CK_X<0$ and the variant $C^2=-2$ is impossible.
%\end{proof}

A weak del Pezzo surface is a del Pezzo surface if and only if it has no $(-2)$-curves. 
Irreducible $(-2)$-curves on $X$ form a subset in $R(X)$ which is a set of simple roots of some root subsystem in $R(X)$. Corresponding positive roots are effective $(-2)$-classes. Two weak del Pezzo surfaces $X$ and $Y$ are said to have the same \emph{type} if there exists an isomorphism $\Pic (X)\to \Pic(Y)$ preserving the intersection form, the canonical class and identifying the sets of negative curves. Thus, weak del Pezzo surfaces are distinguished by the configuration of irreducible $(-2)$-curves. In most cases, the configuration of irreducible $(-2)$-curves determines the type uniquely. Surfaces of different types and with the same configuration of irreducible $(-2)$-curves are distinguished by the number of irreducible $(-1)$-curves. 
A surface of degree $d$ with configuration of $(-2)$-curves $\Gamma$ and with $m$ irreducible $(-1)$-curves is denoted by $X_{d,\Gamma,m}$. Number $m$ is omitted if $d$ and $\Gamma$ determine the type uniquely.
For example, $X_{4,A_2}$ denotes a weak del Pezzo surface of degree $4$ with irreducible $(-2)$-curves forming the diagram $A_2$ and $X_{4,A_3,5}$ denotes a weak del Pezzo surface of degree $4$ with irreducible $(-2)$-curves forming the diagram $A_3$ and with $5$ irreducible $(-1)$-curves. 

\begin{predl}
\label{prop_loslo}
Let $D$ be an $r$-class on a weak del Pezzo surface~$X$ of degree $d$. Then Table~\ref{table_loslo} gives necessary and sufficient condition for $D$ being left-orthogonal and strong left-orthogonal. 
\begin{table}[h]
\label{table_loslo}
\caption{Criteria of left-orthogonality and strong left-orthogonality}
\begin{center}
\begin{tabular}{|c|c|c|}
		 \hline
			$r$ & $D$ is lo? & $D$ is slo?\\
		 \hline
			$\le -3$ & iff $h^0(-D)=0$ & no \\
		 \hline
			$-2$ & iff $h^0(-D)=0$ & iff $h^0(D)=h^0(-D)=0$ \\
		 \hline	
			$-1\le r \le d-3$ & yes & yes \\
		 \hline	
			$\ge d-2$ & iff $h^0(K_X+D)=0$ & iff $h^0(K_X+D)=0$ \\
		 \hline	
\end{tabular}
\end{center}
\end{table}
\end{predl}

\begin{proof}
Check \cite[Proposition 3.2]{EXZ} for details.
\end{proof}

The following consequences of Proposition~\ref{prop_loslo} are very important.
\begin{corollary}
\label{cor_effslo}
On a weak Del Pezzo surface $X$ we have 
$$R=R^{eff}\sqcup(-R^{eff})\sqcup R^{slo};\qquad R^{lo}=R^{eff}\sqcup R^{slo}.$$
\end{corollary}

\begin{predl}
\label{lemma_removeR}
Let $C\subset X$ be a smooth rational $-2$-curve on a weak del Pezzo surface $X$ and $D$ be a divisor on $X$ such that $C\cdot D=1$. Then 
\begin{enumerate}
\item $D$ is left-orthogonal $\Longleftrightarrow$ $D+C$ is left-orthogonal;
\item $D$ is strong left-orthogonal $\Longleftrightarrow$ $D+C$ is strong left-orthogonal.
\end{enumerate}
\end{predl}  
 
\begin{proof}
\begin{enumerate}
\item Twist the short exact sequence 
\begin{equation}
\label{eq_C}
0\to \O_X(-C)\to \O_X\to \O_C\to 0
\end{equation}
by $\O_X(-D)$. Since $C\cdot (-D)=-1$, one gets
$$0\to \O_X(-C-D)\to \O_X(-D)\to \O_C(-1)\to 0.$$
Hence $H^i(O_X(-C-D))\cong H^i(O_X(-D))$ and the statement follows.  
\item Now twist (\ref{eq_C}) by $\O_X(C+D)$. Since $C\cdot(C+D)=-1$, one gets 
$$0\to \O_X(D)\to \O_X(C+D)\to \O_C(-1)\to 0.$$
Thus $H^i(O_X(C+D))\cong H^i(O_X(D))$ and the statement follows.  
\end{enumerate}
\end{proof}

\section{Exceptional collections and toric systems}

Let $X$ be a numerical rational surface, i.e: $h^1(O_X)=h^2(O_X)=0$. Recall that an object $\EE$  in the derived category $\D^b(\coh(X))$ of coherent sheaves on $X$ is said to be \emph{exceptional} if $\Hom^i(\EE,\EE)=0$ for
 all $i\ne 0$ and  $\Hom(\EE,\EE)=\k$. A sequence $(\EE_1,\ldots,\EE_n)$ of exceptional objects is said to be \emph{exceptional} if $\Hom^i(\EE_l,\EE_k)=0$ for
 all $i$ and $k<l$. An exceptional sequence $(\EE_1,\ldots,\EE_n)$  is said to be \emph{strong exceptional} if $\Hom^i(\EE_l,\EE_k)=0$ for
 all $i\ne 0$ and all $k,l$.

Any line bundle on a numerical rational surface is exceptional. Therefore an ordered collection of line bundles $(\O_X(D_1),\ldots,\O_X(D_n))$ on a numerical rational surface $X$ is exceptional (resp. strong exceptional) if and only if the divisor $D_l-D_k$ is left-orthogonal (resp. strong left-orthogonal) for all $k<l$.

Next we recall the important notion of a \emph{toric system}, introduced by Hille and Perling in \cite{HP}.

For a sequence $(\O_X(D_1),\ldots,\O_X(D_n))$ of line bundles one can consider 
the infinite sequence (called a \emph{helix}) $(\O_X(D_i)), i\in \Z$, defined by the rule
$D_{k+n}=D_k-K_X$. From Serre duality it follows that the collection $(\O_X(D_1),\ldots,\O_X(D_n))$ is exceptional (resp. numerically exceptional) if and only if any collection of the form $(\O_X(D_{k+1}),\ldots,\O_X(D_{k+n}))$ is exceptional (resp. numerically exceptional). One can consider the $n$-periodic sequence $A_k=D_{k+1}-D_k$ of divisors on $X$. Following Hille and Perling, we will consider the finite sequence $(A_1,\ldots,A_n)$ with the cyclic order and will treat the index $k$ in $A_k$ as a residue modulo $n$.  Vice versa, for any sequence $(A_1,\ldots,A_n)$ one can construct the infinite sequence $(\O_X(D_i)), i\in \Z$, with the property $D_{k+1}-D_k=A_{k\mod n}$.

\begin{definition}[See {\cite[Definitions 3.4 and 2.6]{HP}}]
A sequence $(A_1,\ldots,A_n)$ in $\Pic(X)$ is called a \emph{toric system} if $n=\rank K_0^{num}(X)$ and the following conditions are satisfied (where indexes are treated modulo $n$):
\begin{itemize}
\item $A_iA_{i+1}=1$;
\item $A_iA_{j}=0$ if $j\ne i,i\pm 1$;
\item $A_1+\ldots+A_n=-K_X$.
\end{itemize}
\end{definition}

Note that a cyclic shift $(A_k,A_{k+1},\ldots,A_n,A_1,\ldots,A_{k-1})$ of a toric system $(A_1,\ldots,A_n)$  is also a toric system.

\begin{example}
Let $Y$ be a smooth projective surface obtained by blowing up three points $P_1,P_2,P_3$ on $\mathbb{P}^2$, where $P_1,P_2,P_3$ are on the same line $L\subset \mathbb{P}^2$. Then $(L-E_{12},E_2,L-E_{23},E_3,L-E_{13},E_1)$ is a toric system on $Y$. In fact, it is a cyclic strong exceptional toric system on $Y$ by Corollary $4.8$.

%Let $Y$ be a smooth projective toric surface. Its torus-invariant prime divisors form a cycle, denote them $T_1,\ldots,T_n$ in the cyclic order. Then $(T_1,\ldots,T_n)$ is a toric system on~$Y$.
\end{example}

The following is proved in \cite[Lemma 3.3]{HP}, see also \cite[Propositions 2.8 and 2.15]{EL}.
\begin{predl}
A sequence $(A_1,\ldots,A_n)$ in $\Pic(X)$ is a toric system if and only if $n=\rank K_0^{num}(X)$ and the corresponding collection $(\O_X(D_1),\ldots,\O_X(D_n))$ is numerically exceptional.
\end{predl}

For the future use we make the following remark on rational surface $X$, see the proof of Proposition 2.7 in~\cite{HP}.
\begin{remark}
For any toric system $(A_1,\ldots,A_n)$ the elements $A_1,\ldots,A_n$ generate $\Pic X$ as abelian group.
\end{remark}

A toric system $(A_1,\ldots,A_n)$ is called \emph{exceptional} (resp. \emph{strong exceptional}) if the corresponding  collection $(\O_X(D_1),\ldots,\O_X(D_n))$ is  exceptional (resp. strong exceptional). Note that exceptional toric systems are stable under cyclic shifts while strong exceptional toric systems are not in general.

A toric system $(A_1,\ldots,A_n)$ is called \emph{cyclic strong exceptional} if the collection 
$$(\O_X(D_{k+1}),\ldots,\O_X(D_{k+n}))$$ 
is strong exceptional for any $k\in \Z$. Equivalently: if all cyclic shifts $$(A_k,A_{k+1},\ldots,A_n,A_1,\ldots,A_{k-1})$$
are strong exceptional.

Notation: for a toric system $(A_1,\ldots,A_n)$ denote 
$$A_{k,k+1,\ldots,l}=A_k+A_{k+1}+\ldots+A_l.$$
We allow $k>l$ and treat 
$[k,k+1,\ldots,l]\subset [1,\ldots, n]$ as a cyclic segment.
Note that $A_{k\ldots l}$ is a numerically left-orthogonal divisor with 
\begin{equation}
\label{eq_Akl2}
A_{k\ldots l}^2+2=\sum_{i=k}^l(A_i^2+2).
\end{equation}

\begin{remark}
If for a toric system $A$ one has $A_i^2\ge -2$ for all $i$, then one has $A_{k,\ldots,l}^2\ge -2$ for any cyclic segment $[k,k+1,\ldots,l]\subset [1,\ldots, n]$. 
\end{remark}

The next proposition is a straightforward consequence of definitions.
\begin{predl}
\label{prop_straightforward}
\begin{enumerate}
\item
A toric system $(A_1,\ldots,A_n)$ is  exceptional if and only if the divisor $A_{k\ldots l}$ is left-orthogonal for all $1\le k<l\le n-1$ and if and only if the divisor $A_{k\ldots l}$ is left-orthogonal for all $k,l\in [1\ldots n], l\ne k-1$.
\item
A toric system $(A_1,\ldots,A_n)$ is strong  exceptional if and only if the divisor $A_{k\ldots l}$ is strong left-orthogonal for all $1\le k<l\le n-1$.
\item
A toric system $(A_1,\ldots,A_n)$ is  cyclic strong exceptional if and only if the divisor $A_{k\ldots l}$ is strong  left-orthogonal for all $k,l\in [1\ldots n], l\ne k-1$.
\end{enumerate}
\end{predl}

An exceptional collection in  $\D^b(\coh(X))$ is said to be \emph{full} if it generates the category  $\D^b(\coh(X))$ as triangulated category. An exceptional collection of $n$ elements is said to be \emph{of maximal length} if $n$ equals to the rank of the numerical Grothendieck group $K_0^{num}(X)$ of $X$. Any full exceptional collection has maximal length, the converse in general is not true. It follows from \cite{KU} that on a weak del Pezzo surface of degree $\ge 2$ any exceptional collection of maximal length is full. 
By our definition, we consider only toric systems that correspond to exceptional collections of maximal length. Recall that for a rational surface $X$ one has
$$\rank K_0(X)=\rank\Pic(X)+2=12-\deg(X).$$

\begin{theorem}
\label{th_checkonlyminustwo}
Let $A=(A_1,\ldots,A_n)$ be a toric system on a weak del Pezzo surface $X$. Suppose $A^2_i\ge -2$ for  $1\le i\le n-1$. 
Then: 
\begin{enumerate}
\item
$A$ is exceptional if and only if the following holds: 
for any cyclic segment $[k\ldots l]\subset [1\ldots n]$ such that one of the next conditions holds:
\begin{enumerate}
\item $1\le k\le l\le n-1$ and  $A_{k\ldots l}^2=-2$,
\item $k>l$ and $A_{k\ldots n\ldots l}^2=A_n^2\le -2$
\end{enumerate}
the divisor $A_{k\ldots l}$ is left-orthogonal (or equivalently: the divisor $A_{k\ldots l}$ is not anti-effective).
\item
$A$ is strong exceptional if and only if $A$ is exceptional and the following holds: for any $1\le k\le l\le n-1$ such that $A_{k\ldots l}^2=-2$, the divisor $A_{k\ldots l}$ is strong left-orthogonal (or equivalently: the divisor $A_{k\ldots l}$ is not effective nor anti-effective).
\end{enumerate}

Suppose moreover that $A^2_i\ge -2$ for  all $i$. Then
\begin{enumerate}
\item[(3)]
$A$ is exceptional if and only if the following holds: for any cyclic segment $[k\ldots l]\subset [1\ldots n]$ such that $A_{k\ldots l}^2=-2$, the divisor $A_{k\ldots l}$ is left-orthogonal (or equivalently: the divisor $A_{k\ldots l}$ is not anti-effective).
\item[(4)]
$A$ is cyclic strong exceptional if and only if the following holds: for any cyclic segment $[k\ldots l]\subset [1\ldots n]$ such that $A_{k\ldots l}^2=-2$, the divisor $A_{k\ldots l}$ is strong left-orthogonal (or equivalently: the divisor $A_{k\ldots l}$ is not effective nor anti-effective).
\item[(5)]
For any cyclic segment $[k\ldots l]\subset [1\ldots n]$, $-2\leq A_{k\ldots l}^2\leq d-2$, where $d=K_X^2$
\end{enumerate}
\end{theorem} 
\begin{proof}
%Recall Corollary~\ref{cor_effslo}: a $(-2)$-class is lo $\Longleftrightarrow$ it is not %anti-effective, a $(-2)$-class is slo $\Longleftrightarrow$ it is not effective nor %anti-effective. 

%Also we note that $\deg X\ge 3$ by  Proposition~\ref{prop_csadm}.
See \cite[Theorem 4.7]{EXZ} for details.

\end{proof}

\begin{corollary}
Let $A=(A_1,\ldots,A_n)$ be a strong exceptional toric system with $A_n^2\ge -1$. Then $A$ is cyclic strong exceptional.
\end{corollary}
\begin{proof}
Use Theorem~\ref{th_checkonlyminustwo}. Clearly, any cyclic segment $[k,\ldots,l]\subset [1,\ldots,n]$ such that $A_{k,\ldots,l}=-2$ lies in $[1,\ldots,n-1]$. Therefore $A_{k,\ldots,l}$ is slo by assumptions.
\end{proof}

\begin{corollary}
Any toric system $(A_1,\ldots,A_n)$  with $A_i^2\ge -2$ for all $i$ on a del Pezzo surface is cyclic strong exceptional. 
\end{corollary}
\begin{proof}
There are no effective and anti-effective $(-2)$-classes on a del Pezzo surface.
\end{proof}

We finish this section with a very important theorem which is due to Hille and Perling \cite[Theorem 3.5]{HP} for rational surfaces  and to C.\,Vial  \cite[theorem 3.5]{V} for the general case.
\begin{theorem}
\label{theorem_HP}
Let $A=(A_1,\ldots,A_n)$ be a toric system of maximal length on a smooth projective surface $X$ such that $\chi(\O_X)=1$ ($X$ is not necessarily rational). Then there exists a smooth projective toric surface $Y$ with torus-invariant divisors $T_1,\ldots,T_n$ such that $A_i^2=T_i^2$ for all $i$.
\end{theorem}

%\begin{predl}
%Let $X$ be a smooth projective surface admitting a full exceptional collection of line bundles. Then $Pic(X)\cong N^1(X)$.
%\end{predl}

%\begin{proof}
%Let $(E_1,\ldots,E_n)$ be a full exceptional collection of line bundles on $X$, then $K_0(X)\cong\mathbb{Z}^n$ since $K_0$ is compatible with semi-orthogonal decomposition of $D^b(X)$. By \cite[Lemma 2.6]{V}, $Pic(X)\cong CH^1(X)$ is torsion free. Consider the exponential exact sequence:

%$$0\rightarrow\mathbb{Z}\rightarrow O_X\rightarrow O_X^*\rightarrow 0 $$.
%The long exact sequence of cohomology of sheaves corresponding to it is:
%$$\ldots\rightarrow0=H^1(X,O_X)\rightarrow H^1(X,O_X^*)\xrightarrow{c_1} H^2(X,\mathbb{Z})\rightarrow H^2(X,O_X)=0 $$

%Thus, $Pic(X)\cong H^1(X,O_X^*)\cong H^2(X,\mathbb{Z})$ is finitely generated since $H^2(X,\mathbb{Z})$ is finitely generated. Thus, $Pic(X)$ is free abelian group. Futher note that $Pic^0(X)=ker(c_1)=0$, so the Neron-Severi group of $X$, $NS(X)\cong Pic(X)$. By definition, $N^1(X)$ is $NS(X)$ modulo torsions, but $NS(X)$ is free abelian group. Hence, $N^1(X)\cong NS(X)\cong Pic(X)$. 
%\end{proof}

%\begin{remark}
%In the case above, we have $rkPic(X)+2=rkN^1(X)+2=rk K_0^{num}(X)=rkK_0(X)$.
%\end{remark}

\section{Admissible sequences}

For a sequence $(a_1,\ldots,a_n)$ of integers we define the \emph{$m$-th elementary augmentation} as follows:

\begin{itemize}
\item $augm_1(a_1,\ldots,a_n) = (-1,a_1-1,a_2,\ldots,a_{n-1},a_n-1)$;
\item $augm_m(a_1,\ldots,a_n)=(a_1,\ldots,a_{m-2},a_{m-1}-1,-1,a_m-1,a_{m+1},\ldots,a_n)$ for $2\le m\le n$; 
\item $augm_{n+1}(a_1,\ldots,a_n)=(a_1-1,a_2,a_3,\ldots,a_n-1,-1)$.
\end{itemize}

We call a sequence  \emph{admissible} if it can be obtained from the sequence $(0,k,0,-k)$ or $(k,0,-k,0)$ by applying several elementary augmentation.
It is not hard to see that admissible sequences are stable under cyclic shifts:
$$sh(a_1,\ldots a_n)=(a_2,\ldots,a_m,a_1)$$
and symmetries:
$$sym(a_1,\ldots a_n)=(a_{n-1},a_{n-2},\ldots,a_1,a_n).$$
Also note that for an admissible sequence $(a_1,\ldots,a_n)$ one has 
$$\sum_{i=1}^na_i=12-3n.$$
We call a sequence $a_1,\ldots,a_n$ of integers \emph{strong admissible} if it is admissible and $a_i\ge -2$ for $1\le i\le n-1$. 
We call a sequence $a_1,\ldots,a_n$ of integers \emph{cyclic strong admissible} if it is admissible and $a_i\ge -2$ for $1\le i\le n$. 
Note that cyclic strong admissible sequences are stable under cyclic shifts while strong admissible are not. Strong admissible and cyclic strong admissible sequences are stable under symmetries.

For a toric system $A=(A_1,\ldots,A_n)$ on $X$, we denote 
$$A^2=(A_1^2,\ldots,A_n^2).$$

The motivation for considering admissible sequences is the following. Let $Y$ be a toric surface with torus invariant divisors $T_1,\ldots,T_n$. Then the sequence $T^2=(T_1^2,\ldots,T_n^2)$ is admissible. Indeed, for $Y$ a Hirzebruch surface one has $n=4$ and the statement is clear. Otherwise $Y$ has a torus-invariant $(-1)$-curve $E$. Let $E=T_k$, and consider the blow-down $Y'$ of $E$. The torus-invariant divisors on $Y'$ are $T'_1,\ldots,T'_{k-1},T'_{k+1},\ldots,T'_n$, and $(T'_{k\pm 1})^2=T_{k\pm 1}^2+1$, $(T'_{i})^2=T_{i}^2$ otherwise. Therefore $T^2=augm_k((T')^2)$, and we proceed by induction.

Further, from Theorem~\ref{theorem_HP} it follows now that for any toric system $A$ the sequence $A^2$ is admissible. Suppose that $A$ is a strong exceptional toric system, then
for $1\le i\le n-1$ the divisor $A_i$ is slo, hence $A_i^2=\chi(A_i)-2\ge -2$ and the sequence $A^2$ is strong admissible. The same holds for cyclic strong. 

The next Proposition can be checked directly or deduced from \cite[Table 1]{HP}.

\begin{predl}
\label{prop_csadm}
All cyclic strong admissible sequences are, up to cyclic shifts and symmetries, in  Table~\ref{table_csadm}.
\begin{table}[h]
\caption{Cyclic strong admissible sequences}
\begin{center}
\begin{tabular}{|c|c|}
		 \hline
			type & sequence\\
		 \hline
			$\P^1\times \P^1$ & $(0,0,0,0)$ \\
		 \hline	
			$F_1$ & $(0,1,0,-1)$ \\
		 \hline	
			$F_2$ & $(0,2,0,-2)$ \\
		 \hline	
			5a & $(0,0,-1,-1,-1)$ \\
		 \hline	
			5b & $(0,-2,-1,-1,1)$ \\
		 \hline	
			6a & $(-1,-1,-1,-1,-1,-1)$ \\
		 \hline	
			6b & $(-1,-1,-2,-1,-1,0)$ \\
		 \hline	
			6c & $(-2,-1,-2,-1,0,0)$ \\
		 \hline	
			6d & $(-2,-1,-2,-2,0,1)$ \\
		 \hline	
			7a & $(-1,-1,-2,-1,-2,-1,-1)$ \\
		 \hline	
			7b & $(-2,-1,-2,-2,-1,-1,0)$ \\
		 \hline	
			8a & $(-2,-1,-2,-1,-2,-1,-2,-1)$ \\
		 \hline	
			8b & $(-2,-1,-1,-2,-1,-2,-2,-1)$ \\
		 \hline	
			8c & $(-2,-1,-2,-2,-2,-1,-2,0)$ \\
		 \hline	
			9 & $(-2,-2,-1,-2,-2,-1,-2,-2,-1)$ \\
		 \hline	
\end{tabular}
\end{center}
\label{table_csadm}
\end{table}

In particular, if a surface $X$ has a toric system $A=(A_1,\ldots,A_n)$ with $A_i^2\ge -2$ for all~$i$ then $n\le 9$ and $\deg(X)\ge 3$.
\end{predl}

\begin{predl}
\label{prop_ncsadm}
Any strong admissible sequence $(a_1,\ldots,a_n)$ which is not cyclic strong admissible is, up to a symmetry, in Table~\ref{table_ncsa}.
\begin{table}[h]
\caption{Strong admissible sequences which are not cyclic strong}
\begin{center}
\begin{tabular}{|c|c|}
\hline
type & sequence\\
\hline
IIa & $(b,c,d,e)$,  $c+e=4-n$; \\
IIb & $(-2,-1,-2,c,d,e)$, $c+e=5-n$; \\
IIc & $(-2,-1,-2,c,-2,-1,-2,e)$, $c+e=6-n$; \\
& where $c,e\in \Z$,  $c\ge -2$, $e\le -3$\\
& and $b$ and $d$ are sequences of the form $(0),(-1,-1)$ or $(-1,-2,-2,\ldots,-2,-1)$; \\
\hline	
IIIa & $(1,0,-1,-2,\ldots,-2,-1,4-n)$; \\
IIIb & $(-1,0,0,-2,\ldots,-2,-1,4-n)$; \\
IIIc & $(-1,-2,\ldots,-2,0,0,-2,\ldots,-2,1,4-n)$; \\
\hline	
IV & $(-2,0,1,-2,\ldots,-2,-1,4-n)$; \\
\hline	
V & $(-2,-1,-1,0,-2,\ldots,-2,-1,5-n)$; \\
\hline	
VI & $(-2,-2,-1,-2,0,\ldots,-2,-1,6-n)$. \\
\hline	
\end{tabular}
\end{center}
\label{table_ncsa}
\end{table}
\end{predl}

\section{Anticanonical Pseudo height of exceptional collection of line bundles}
A substantial amount of work was carried out to understand the phantom category $\mathcal{A}$ (subcategory of $D^b(X)$ with trivial $K_0$ and $HH_0$) in the derived category of coherent sheaves $D^b(X)$ on smooth projective surfaces of general type with $p_g=q=0$. In principle, in \cite{Kuz}, A.Kuznetsov gives an algorithm to compute $HH^*(\mathcal{A})$ using a spectral sequence and the notion of height of an exceptional collection. Moreover by \cite{Kuz}, if an exceptional collection is full, its height must vanish. Thus the height may be used to prove the existence of the phantom category. But in general, the height of an exceptional collection is hard to compute. Instead, in \cite{Kuz}, A.Kuznetsov suggests a coarse invariant which is easy to compute and is strong enough for our purpose. Let us recall the definition:

\begin{definition}\cite[Definition 4.9]{Kuz}
Let $\mathbb{E}=(E_1,\ldots,E_n)$ be an exceptional collection of maximal length on $X$. The anticanonical pseudoheight of $\mathbb{E}$ is defined as: 

$$ ph_{ac}(\mathbb{E})=min_{1\leq a_0<a_1<\ldots<a_p\leq n}(e(E_{a_0},E_{a_1})+\ldots+e(E_{a_{p-1}},E_{a_p})+e(E_{a_p},E_{a_0}\otimes\omega_X^{-1})-p). $$
where $e(E,E')=min\{k|Ext^k(E,E')\neq 0\}$. 
\end{definition}

Suppose $\mathbb{E}$ consists of only line bundles on $X$ and the corresponding exceptional toric system is given by $\mathcal{A}=(A_1,\ldots,A_n)$. Then the definition of the anticanonical pseudo height becomes:
$$ ph_{ac}(\mathcal{A})=min_{1\leq a_0<a_1<\ldots<a_p\leq n}(e(A_{a_0}+\ldots+A_{a_1-1})+\ldots+e(A_{a_{p-1}}+\ldots+A_{a_p-1})+e(-K_X-A_{a_0}-\ldots-A_{a_{p}-1})-p). $$

Where $e(A_k+\ldots+A_l)=min\{q|h^q(A_k+\ldots+A_l)\neq 0\}$ and $1\leq p\leq n-1$.

Note that if $A_k+\ldots+A_l$ is an effective divisor, $e(A_k+\ldots+A_l)=0$, otherwise $e(A_k+\ldots+A_l)\geq 1$. 

\begin{lemma}\cite[Corollary 6.2]{Kuz}
Let $\mathbb{E}=(E_1,\ldots,E_n)$ be a full exceptional collection on a smooth projective surface $X$. Then the anticanonical pseudo height of $\mathbb{E}$ satisfies $ph_{ac}(\mathbb{E})\leq -2$.
\end{lemma}

\begin{lemma}
Let $\mathcal{A}$ be a full exceptional toric system on a smooth projective surface $X$. Then there must exist an effective divisor among divisors of the form $A_k+\ldots+A_l$, where $[k,l]$ is a (cyclic)segment of length less or equal to $n-1$.
\end{lemma}

\begin{proof}
By assumption, there exists a full exceptional collection of line bundles $\mathbb{E}$ on $X$. By Lemma $6.2, ph_{ac}\leq -2$. Assume that there does not exist any effective divisor of the form $A_k+\ldots+A_l$. Then $h^0(\sum_{k}^{l}A_j)=0$ for all $k\leq l$(resp. $k>l$) in an interval (resp. cyclic interval) of length $\leq n-1$. Thus, $e(E_{a_i},E_{a_{i+1}})\geq 1, e(E_{a_p},E_{a_0}\otimes\omega_X^{-1})\geq 1$ for all $a_0\leq i\leq a_{p-1}$. Hence $ph_{ac}(\mathbb{E})\geq (p+1)-p=1$, contradiction.
\end{proof}

\section{Rationality of smooth projective surfaces}
There is a folklore conjecture of D.Orlov on the rationality of projective algebraic surfaces:

\begin{conjecture}[Orlov]
If $X$ is a smooth projective surface admitting a full exceptional collection, then $X$ is rational surface.

\end{conjecture}

The paper \cite{BS} proved this conjecture under the condition that the full exceptional collection consists of only line bundles and the collection is in additional strong:

\begin{theorem}\cite[Theorem 4.4]{BS}
Let $X$ be a smooth projective surface with a full strong exceptional collection of line bundles. Then $X$ must be a rational surface.
\end{theorem}

By carefully investigating their proof, it looks like the fullness condition is not necessary, instead we prove the following statement:

\begin{theorem}
Let $X$ be a smooth projective surface with strong exceptional collection of line bundles of maximal length, then $X$ must be a rational surface.
\end{theorem}

\begin{proof}
Let $\mathbb{E}:=(O(D_1),\ldots,O(D_n))$ be a strong exceptional collection of line bundles of maximal length on $X$. The associated strong exceptional toric system $\mathcal{A}$ is given by: $A_1:=D_2-D_1,\ldots,A_i=D_{i+1}-D_i,\ldots,A_{n-1}=D_n-D_{n-1}, A_n=-K_X-\sum_{i=1}^{n-1}A_i$.
Thus, $A_i$ is slo divisors for $1\leq i\leq n-1$, hence the sequences of $A^2$ are classified as strong admissble sequences shown in table $3$ and table $4$. Thus, there exists a slo divisor $D$ of the form $A_i+A_{i+1}+\ldots+A_j, 1\leq i\leq j\leq n-1$ such that $D^2\geq 0$, thus $h^0(D)=D^2+2\geq 2$. Thus $|D|$ has a moving component. Then, we borrow part of the argument in \cite[Theorem 4.4]{BS} to conclude that $X$ is covered by smooth rational curves. Then, the Kodaira dimension of $X$, $kod(X)=-\infty$. By the classification of surfaces, $X$ is a blow up of a ruled surface $Y$ over a curve $C$. We claim that $C$ is also a rational curve. Indeed, there is a chain of maps: $X\xrightarrow{\pi}Y\xrightarrow{p} C$, where $\pi$ is a blow up and $p$ is a projection. $h^1(O_Y)=h^1(O_C)=p_a(C)$. Thus,
$h^1(O_X)=h^1(O_Y)=p_a(C)$. Since $X$ admits strong exceptional collection of line bundles, $h^1(O_X)=0$. Thus $p_a(C)=0$. Thus, $C$ is a rational curve and $Y$ is Hirzebruch surface, and $X$ is a rational surface.
\end{proof} 

\begin{corollary}
If $X$ is a smooth surface of general type with $p_q=q=0$, then it does not admit strong exceptional collections of line bundles of maximal length.
\end{corollary}

Although it seems that Orlov's folklore conjecture is out of reach at the moment, we can still prove the following statement.
\begin{theorem}
Let $X$ be a smooth projective surface with $rk N^1(X)\leq 3$ with a full exceptional collection of line bundles. Then $X$ must be a rational surface.
\end{theorem}

First, we give several lemmas. 

The next lemma gives the classification of admissible sequences of length $\leq 5$.
\begin{lemma}
All admissible sequences of length $\leq 5$ are given in Table $15$ (up to cyclic shift).
\begin{table}[h]
\caption{admissible sequences of length $\leq 5$}
\begin{center}
\begin{tabular}{|c|c|}
		 \hline
			type & sequence\\
		 \hline
			$\mathbb{P}^2$ & $(1,1,1)$ \\
		 \hline	
			$\mathbb{F}_m, m\in\mathbb{Z}$ & $(m,0,-m,0)$ \\
		 \hline	
			$5s, s\in\mathbb{Z}$ & $(-1,s,0,-s-1,-1)$ \\
		 \hline	
		
\end{tabular}
\end{center}
\label{table_csa}
\end{table}

\end{lemma}

\begin{proof}
By Theorem 4.10, the admissible sequences of length $\leq 5$ corresponds to all toric surfaces of Picard rank $\leq 3$, which can be checked directly. 
\end{proof}

\begin{lemma}
Let $D$ be an effective left-orthogonal divisor such that $D^2=r\geq 0$ on a smooth projective surface $X$ with $h^1(O_X)=h^2(O_X)=0$, then $|D|$ has a moving component. 

%which is a rational curve generically. 
\end{lemma}
\begin{proof}
By the Riemann-Roch theorem on surfaces:
$$ \chi(D)=h^0(D)-h^1(D)+h^2(D)=\frac{1}{2}(D^2-D.K_X)+\chi(O_X) $$
 
$\chi(O_X)=1$,  $D$ is left-orthogonal divisor, thus it is of course a numerical left-orthogonal divisor: $\chi(-D)=0: D^2+D.K_X=-2$. Then, $\chi(D)=r+2\geq 2$. I claim that $h^2(D)=0$. Indeed, consider the standard exact sequence:
 $$0\rightarrow O(-D)\rightarrow O_X\rightarrow O_D\rightarrow 0 $$
 Twisting by the line bundle $O(K_X)$:
 $$0\rightarrow O(K_X-D)\rightarrow O(K_X)\rightarrow O_D(K_X)\rightarrow 0.$$
 
Thus, we have an injective map: $h^0(K_X-D)\rightarrow h^0(K_X)$. But $h^0(K_X)=h^2(O_X)=0$, then $h^0(K_X-D)=0$. Thus $h^2(D)=h^0(K_X-D)=0$. Thus, $h^0(D)-h^1(D)=r+2$. Hence $h^0(D)=r+2+h^1(D)\geq 2$. Thus $|D|$ has a moving component. 

%denoted by $M$. We claim that $M$ is a smooth rational curve. In fact, by Elagin's criterion\cite[Theorem 1.1]{E}, as a connected component of $D$, $p_a(M)\leq 0$ and By Bertini's theorem, generic $M$ is smooth and irreducible, so $p_a(M)=0$. Thus, $M$ is a smooth rational curve.
\end{proof}

\begin{proof}[Proof of Theorem 7.5(or 1.3)]
Let $X$ admit a full exceptional collection of line bundles(in particular, $h^1(O_X)=h^2(O_X)=0$) of length $n$, then $K_0(X)\cong\mathbb{Z}^n$, where $n=rk K_0(X)=rk N^1(X)+2\leq 5$. Let $\mathcal{A}=(A_1,\ldots,A_n)$ be the corresponding full exceptional toric system. Then by Theorem $4.10$, the sequence $\mathcal{A}^2:=(A_1^2,\ldots.A_n^2)$ has to be one in table $5$. We claim that among divisors of the form $A_k+\ldots+A_l$, where $k,l$ are in the (cyclic)interval of length $\leq n-1$, there must exist an effective left-orthogonal divisor $D$ with $D^2\geq 0$. Then by Lemma $7.7$, the linear system $|D|$ has a moving component. Note that $D$ is an effective left-orthogonal divisor, so by Elagin's criterion\cite[Theorem 4.5, Proposition 4.3]{E}, $D=\sum_{i}k_iD_i$ with $k_i\geq 0$ and $D_i$ is a smooth rational curve for each $i$. Thus, $X$ is covered by smooth rational curves (as in the proof of \cite[Theorem 4.4]{BS}), so we conclude that $X$ is a rational surface as in the proof of Theorem $7.3$. Now, we prove the claim:

\begin{enumerate}
\item If $\mathcal{A}^2=(1,1,1)$, by Lemma $6.3$, we have an effective lo divisor $D^2\geq 1$.
\item If $\mathcal{A}^2=(0,m,0,-m)$, $n=4$, thus $1\leq p\leq 3$. We claim that there must exist at least three effective divisors among $A_{a_0}+\ldots+A_{a_1-1},\ldots,A_{a_{p-1}}+\ldots+A_{a_p-1},-K_X-A_{a_0}-\ldots-A_{a_{p}-1}$ for some $p$ and $(a_0,a_1,\ldots,a_{p})$ in the toric system $\mathcal{A}:=(A_1,A_2,A_3,A_4)$. Otherwise, suppose there exist at most $2$ effective divisors among them for any $p$ and any $(a_0,a_1,\ldots,a_p)$. Then we show that $ph_{ac}(\mathcal{A})\geq -1$ as follows:

\begin{enumerate}
\item $p=1$, $P:=(e(A_{a_0}+\ldots+A_{a_1-1})+e(-K_X-A_{a_0}-\ldots-A_{a_{p}-1})-1)\geq -1$, 
\item $p=2$, $Q:=e(A_{a_0}+\ldots+A_{a_1-1})+e(A_{a_1}+\ldots+A_{a_2-1})+e(-K_X-\ldots-A_{a_{p}-1})-2$. By assumption, there are at most $2$ effective divisors among $A_{a_0}+\ldots+A_{a_1-1}, A_{a_1}+\ldots+A_{a_2-1}$ and $-K_X-A_{a_0}-\ldots-A_{a_{p}-1}$. Thus, $Q\geq -1$.
\item $p=3$, $S:=e(A_1)+e(A_2)+e(A_3)+e(A_4)-3$. By assumption, at most two of $A_1,A_2,A_3,A_4$ are effective divisors. Thus, $S\geq -1$. 
\end{enumerate}

Thus, $ph_{ac}(\mathcal{A})=min\{P,Q,S\}\geq -1$. But the toric system $\mathcal{A}$ is full exceptional toric system, $ph_{ac}(\mathcal{A})\leq -2$ by Lemma $6.2$, so we have a contradiction. Thus, there must exist at least three effective divisors among $A_{a_0}+\ldots+A_{a_1-1},\ldots,A_{a_{p-1}}+\ldots+A_{a_p-1},-K_X-A_{a_0}-\ldots-A_{a_{p}-1}$. Then it is easy to see that there must exist at least one effective left-orthogonal divisor $D$ with $D^2\geq 0$. 

\item $\mathcal{A}^2=(-1,n,0,-n-1,-1)$, $1\leq p\leq 4$. A similar argument implies that there exist at least $3$ effective divisors among $A_{a_0}+\ldots+A_{a_1-1},\ldots,A_{a_{p-1}}+\ldots+A_{a_p-1},-K_X-A_{a_0}-\ldots-A_{a_{p}-1}$. By playing combinatoric games on partitions of $(-1,n,0,-n-1,-1)$ by $(a_0,a_1,\ldots,a_p)$, there is at least one effective divisor $D$ among $A_{a_0}+\ldots+A_{a_1-1},\ldots,A_{a_{p-1}}+\ldots+A_{a_p-1},-K_X-A_{a_0}-\ldots-A_{a_{p}-1}$ with $D^2\geq 0$

\end{enumerate}

\end{proof}

\begin{corollary}
Let $X$ be a smooth projective surface, then $X$ is Hirzebruch surface if and only if $X$ admits a full exceptional collection of line bundles of length $4$. 
\end{corollary}

\subsection{Passing to higher dimension}
In this section, we define toric system on some special higher dimension smooth projective varieties. Let $X$ be a smooth projective variety with $Pic(X)\cong\mathbb{Z}$. Let $\mathbb{E}:=(E_1,\ldots,E_n,E_{n+1},E_{n+2})$ be a full exceptional collection on $X$ such that $E_1,\ldots,E_{n}$ are line bundles and $E_{n+1},E_{n+2}$ are two coherent sheaves on $X$. Since $Pic(X)=\mathbb{Z}H$ for an ample line bundle $O_X(H)$. Then $E_i=O_X(a_iH)$ for $1\leq i\leq n$. Define a system of divisors $\mathcal{A}$ as follows:
$$A_i=
\begin{cases}
(a_{i+1}-a_i)H\quad\text{for}\quad 1\le i\le n-1,\\
c_1(E_{n+1})-c_1(E_n)\quad\text{for}\quad i=n,\\
c_1(E_{n+2})-c_1(E_{n+1})\quad\text{for}\quad i=n+1,\\
-K_X-(A_1+\ldots+A_{n+1})\quad\text{for}\quad i=n+2.
\end{cases}
$$

We call $\mathcal{A}:=(A_1,\ldots,A_{n+2})$ the toric system associated to the exceptional collection $\mathbb{E}:=(E_1,\ldots,E_{n+2})$ on $X$. In \cite[Th 1.2]{Li}, Li proved the following statement:
\begin{theorem}
Let $X$ be a smooth projective variety of dimension $n$($n\geq 3$ and $n\neq 4$). Suppose that there is a full exceptional collection $\mathcal{C}$ of $D^b(X)$ which consists of coherent sheaves. Assume that the length of $\mathcal{C}$ is $n+2$ and $\mathcal{C}$ contains a sub-collection $\{E_1,\ldots,E_n\}$ where $E_i(1\leq i\leq n)$ are line bundles. Then $X$ is an even-dimensional variety and $X$ is either isomorphic to $Q^n$ or is of general type. If in additional, the subcollection $\{E_1,\ldots,E_n\}$ is strong, then $X$ is isomorphic to $Q^n$.
\end{theorem}

It is a little unexpected that the general type variety was not excluded by assuming the fullness of exceptional collection on $X$. If we assume the Orlov's folklore conjecture, $X$ should be a rational variety, the general type varieties can be excluded automatically. Indeed, we are able to exclude the general type varieties by applying the techniques of anticanonical pseudo height of exceptional collection(or corresponding toric system) as in the proof of Theorem 7.5. Thus, we improve Li's theorem as follows:

\begin{theorem}
Let $X$ be a smooth projective variety of dimension $n$($n\geq 3$ and $n\neq 4$). Suppose that there is a full exceptional collection $\mathcal{C}$ of $D^b(X)$ which consists of coherent sheaves. Assume that the length of $\mathcal{C}$ is $n+2$ and $\mathcal{C}$ contains a sub-collection $\{E_1,\ldots,E_n\}$ where $E_i(1\leq i\leq n)$ are line bundles. Then $X$ is an even-dimensional variety and $X$ is isomorphic to $Q^n$.
\end{theorem}

\begin{proof}
It is sufficient to exclude the $(2),(3),(4)$ sequence of $(a_1,\ldots,a_n)$ in the following tables copied from \cite[Th 1.2]{Li}. Thus, the only possibility left is $X$ to be quadrics.
$$\begin{tabular}{|c|c|c|c|c|}
  \hline
  &$(a_1,\cdots,a_n)$ & roots of $P(a)$ & $\lambda$ & $X $\\
  \hline
  (1)& $(a_1,a_1+1,\cdots,a_1+(n-1))$ & $\{-1,-2,\cdots,-(n-1)\}$ & $n$ & $Q^n$\\
  \hline

 (2)&$(a_1,a_1-1,\cdots,a_1-(n-1))$ &  $\{1,2,\cdots,(n-1)\}$ & $-n$ & general type\\
  \hline
  (3)&$(a_1,a_1-1,\cdots,a_1-(n-1))$ & $\{1,2,\cdots,n\}$ & $-n-1$ & general type \\
  \hline

 (4)&$(a_1,\cdots,\widehat{a_1-k},\cdots,a_1-n)$ & $\{1,2,\cdots,n\}$ & $-n-1$ & general type\\
  \hline
\end{tabular}$$

For sequence $(2)$, the corresponding full exceptional collection is $\mathbb{E}:=(O_X(a_1H),O_X((a_1-1)H),\ldots,O_X((a_1-(n-1))H),S_1,S_2)$, where $S_1,S_2$ are two coherent sheaves on $X$. Thus the corresponding full exceptional toric system is given by $\mathcal{A}:=(-H,-H,\ldots,-H,B_1,B_2,C_1)$, where $B_1=c_1(S_1)-(a_1-(n-1))H, B_2=c_1(S_2)-c_1(S_1), C_1=-K_X-(A_1+\ldots+A_{n+1})$. Note that any positive multiple of $H$ is an ample divisor, thus $h^0(O_X(-mH))=0, m>0$ by Kodaira vanishing theorem. Hence $-mH$ is not effective divisor for any $m>0$ and $e(A_i+\ldots+A_j)=e(-(j-i+1)H)\geq 1$ for any $j>i$. By playing the combinatoric game on partitions of the toric system $\mathcal{A}$ as in the proof of Theorem 7.5, it is not hard to see that the anti-canonical pseudo height of $\mathcal{A}$ or $\mathbb{E}$:

$$ph_{ac}(\mathcal{A})\geq (n-1)-p=(n-1)-(n+1)=-2$$ Where $p$ is equal to length of full exceptional collection substract by $1$ and the minimal value of $ph_{ac}$ is reached by assuming every divisor containing one of $B_1,B_2,C_1$ is effective(or $e(E_k,E_{k'})=0$ for $k$ or $k'\in\{n+1,n+2\}$ and $e(E_{a_p},E_{a_0}\otimes\omega^{-1}_X)=0$).On the otherhand, by \cite[Corollary 6.2]{Kuz}, $ph_{ac}(\mathbb{E})\leq -n$ since $\mathbb{E}$ is a full exceptional collection. By assumption, $n\geq 3$, there is a contradiction. Thus the full exceptional collection corresponding to $(2)$ is not possible. 

For sequence $(3),(4)$, almost the same argument will apply, we omit the details here and the full exceptional collection corresponding to these two sequences are either impossible.
\end{proof}

\begin{remark}
By Kodaira vanishing theorem, $h^q(O_X(-mH))=0$ for every $q\leq n-1$, thus $e(-(j-i+1)H)\geq n$. So $ph_{ac}(\mathbb{E})$ is usually very positive, thus $\mathbb{E}$ can not be full in the case of $(2),(3),(4)$. 
\end{remark}

\section{Equivalences of various notions of cyclic strong exceptional collection of line bundles}

In this section, we prove Theorem $1.8$. Let us recall the definition of tilting bundles on a smooth projective surface $X$.

\begin{definition}
Let $D^b(X)$ be a bounded derived category of coherent sheaves on a smooth projective surface $X$. An object, $T$, of $D^b(X)$ is called a tilting object if the following two conditions hold:\begin{enumerate}
    \item $Hom(T,T[i])=0$ for all $i\neq 0$;
    \item $T$ generates $D^b(X)$.
\end{enumerate}

$T$ is called a tilting bundle, if in addition, $T$ is a locally free sheaf on $X$
\end{definition}

The proof is divided into the following propositions:

\begin{predl}
A tilting bundle $T$ on $X$ is pull back iff $T$ is $2$-hereditary. In particular, the pull back exceptional collection of line bundles coincides with the $2$-hereditary exceptional collection of line bundles. 
\end{predl}

\begin{proof}
The proof is borrowed from \cite[Lemma 3.19]{BF}:
We have $\pi_*O_{Tot(\omega_X)}=\bigoplus_{i\geq 0}\omega_X^{-i}$. $\pi^*(T)$ is tilting bundle on $Tot(\omega_X)$ means that:
$$ Hom_{Tot(\omega_X)}(\pi^*T,\pi^*T[k])\cong Hom_X(T,\pi_*\pi^*(T)[k])\cong Hom_X(T,T\otimes\pi_*O_{Tot(\omega_X)}[k])$$
$$\cong\bigoplus_{i\geq 0}Hom_X(T,T\otimes\omega_X^{-i}[k])=0 $$.

Then we have:
$$Ext^k(T,T\otimes\omega_X^{-i})=0$$ for $k\geq 1$ and $i\geq 0$.

Thus, $T$ is $2$-hereditary tilting bundle.
\end{proof}

\begin{predl}
Let $T$ be a tilting bundle on smooth projective surface $X$ which is a direct sum of an exceptional collection of line bundles. Then $T$ is almost pull back iff $T$ is cyclic strong. In particular, almost pull back exceptional collection of line bundles coincides with cyclic strong exceptional collection of line bundles.
\end{predl}

\begin{proof}
One direction is trivial: Let $T$ be an almost pull back tilting bundle formed by a direct sum of an exceptional collection of line bundles: $(E_0,\ldots,E_n)$. Then by Definition $1.5$, $T$ is cyclic strong. 
Now, Let $(E_0,\ldots,E_n)$ be a cyclic strong exceptional collection of line bundles. We can assume that $E_0=O_X,E_1=O_X(A_1),E_2=O(A_1+A_2),\ldots, E_n=O(A_1+A_2+\ldots+A_n)$, where $A_1,A_2,\ldots,A_{n+1}$ is a cyclic strong exceptional toric system on $X$. Then $D:=\sum_{I\subset [n]}A_I$ is strong left-orthogonal divisor for any (cyclic) interval $I$ with length $l\leq n$. It is enough to verify the following three statements:\begin{enumerate}
\item $H^i(X,O(-D-K_X)=0,i\geq 1$
\item $H^i(X,O(D-K_X)=0,i\geq 1$
\item $H^i(X_X,O(-K_X))=0, i\geq 1$
\end{enumerate}

For $(1)$, since $-K_X-D=-K_X-\sum_{I\subset [n]}A_I=\sum_{J\subset [n]}A_J$, where $J$ is also a (cyclic) interval of length $l\leq n$. Thus $-K_X-D$ is also a slo divisor, so $H^i(X,O(-D-K_X)=0, i\geq 1$. For $(3)$: $X$ is a weak del Pezzo surface with $d=K_X^2\geq 3$ by Theorem $15.2$ in \cite{EXZ}. By Lemma $8.3.1$ in \cite{Do}, $H^i(X,O(-K_X))=0,i\geq 1$.

For $(2)$: we consider the standard exact sequence associated to a general member $C\in |-K_X|$, where $C$ is an elliptic curve since $|-K_X|$ is base point free. 
$$0\rightarrow O(K_X)\rightarrow O_X\rightarrow O_C\rightarrow 0 $$
Twisting by $O(D-K_X)$, we get:
$$0\rightarrow O(D)\rightarrow O(D-K_X)\rightarrow O_C(D-K_X)\rightarrow 0 $$

Thus we have long exact sequence of cohomology of line bundles:
$$0\rightarrow H^0(O(D))\rightarrow H^0(O(D-K_X))\rightarrow H^0(O_C(D-K_X))$$
$$\rightarrow H^1(O(D))\rightarrow H^1(O(D-K_X))\rightarrow H^1(O_C(D-K_X))\rightarrow H^2(O(D))$$
$$\rightarrow H^2(O(D-K_X)\rightarrow 0 $$

$H^1(O(D))=H^2(O(D))=0$ since $D$ is slo. Thus, it is enough to show that $H^1(O_C(D-K_X)=0$. $D$ slo divisor, so $D^2+D.K_X=-2$. Note that $C.(D-K_X)=-K_X.D+K_X^2=D^2+2+K_X^2$. By Theorem $4.7, (5)$, $D^2\geq -2$, so $C.(D-K_X)\geq K_X^2\geq 3$. Thus, $deg(O_C(K_X-D))=C.(K_X-D)=-K_X^2\leq -3$, so by Serre duality,  $H^1(C,O_C(D-K_X))=H^0(C,O_C(K_X-D))=0$. Therefore, $H^1(O(D-K_X))=0$ and $H^2(O(D-K_X))=0$

\end{proof}

\begin{predl}
Let $T$ be a tilting bundle on $X$ which is direct sum of an exceptional collection of line bundles. Then $T$ is pull back iff $T$ is cyclic strong. In particular, the pull back exceptional collection of line bundles coincides with the cyclic strong exceptional collection of line bundles.
\end{predl}

\begin{proof}
One direction is trivial: If $T$ is direct sum of exceptional collection of line bundles: $(E_0,\ldots,E_n)$. By proof of Proposition $8.2$: $$Ext^k(E_i,E_j\otimes\omega_X^{-r})=0$$ for all $i,j$ and $r\geq 0$, $k\geq 1$. By definition $1.5$, $T$ is cyclic strong. On the other hand, if $T$ is cyclic strong. we can assume that $E_0=O_X,E_1=O_X(A_1),E_2=O(A_1+A_2),\ldots, E_n=O(A_1+A_2+\ldots+A_n)$, where $\mathcal{A}=(A_1,A_2,\ldots,A_{n+1})$ is a cyclic strong exceptional toric system on $X$. Then $D:=\sum_{I\subset [n]}A_I$ is a strong left-orthogonal divisor for any (cyclic) interval $I$ with length $l\leq n$. To show that 
$$ Ext^k(E_i,E_j\otimes\omega_X^{-r})=0$$ for all $k\geq 1$, all $i,j$ and all $r\geq 0$. In the language of toric system, it is enough to verify that:
$$ H^k(D-rK_X)=0 $$ for all $k\geq 1, r\geq 0$, where $D$ is either trivial, $D$ or $-D$ is a slo divisor. There are three cases:\begin{enumerate}
    \item $D=0$, $H^k(D-rK_X)=H^k(-rK_X)=0$ since $X$ is weak del Pezzo surface by\cite[Theorem $1.3$]{EXZ};
    \item $D$ is slo divisor. \begin{enumerate}
        \item $r=0$, $H^k(D-rK_X)=H^k(D)=0$ for $k\geq 1$ since $D$ is slo. 
        \item $r\geq 1$, we prove by the induction: if $r=1$, $H^k(D-rK_X)=H^k(D-K_X)$. Consider the standard exact sequence:
        $$ 0\rightarrow O(K_X)\rightarrow O_X\rightarrow O_C\rightarrow 0 $$ Where $C\in |-K_X|$ is smooth irreducible curve.
        Twisting by $O(D-K_X)$: 
        $$ 0\rightarrow O(D)\rightarrow O(D-K_X)\rightarrow O_C(D-K_X)\rightarrow 0 $$
        Then we have long exact sequence of cohomology of line bundles:
        $$ \ldots\rightarrow H^1(D)\rightarrow H^1(D-K_X)\rightarrow H^1(O_C(D-K_X))\rightarrow H^2(D)\rightarrow H^2(D-K_X)\rightarrow 0 $$
        $H^1(D)=H^2(D)=0$ since $D$ is slo. $D$ is $R$-class. This means that $D^2+D.K_X=-2$ and $D^2=R$, so $-D.K_X=D^2+2=R+2$. $deg_C(O_C(D-K_X))=-K_X.(D-K_X)=R+2+K_X^2$. Note that $K_X^2\geq 3$ and $R=D^2\geq -2$. Thus, $H^1(O_C(D-K_X))=0$, hence $H^1(D-K_X)=0$. 
        \item Assume that $H^k(D-(r-1)K_X)=0$. Consider the standard exact sequence twisted by $O_X(D-rK_X)$: 
        $$ 0\rightarrow O_X(D-(r-1)K_X)\rightarrow O_X(D-rK_X)\rightarrow O_C(D-rK_X)\rightarrow 0 $$  The long exact sequence of cohomology of line bundles is:
        $$ \ldots\rightarrow H^k(O_X(D-(r-1)K_X)\rightarrow H^k(D-rK_X)\rightarrow H^k(O_C(D-rK_X))\rightarrow\ldots $$
        Thus, $H^k(D-rK_X)=0$ for all $k\geq 1$ and $r\geq 2$ since $deg(O_C(D-rK_X))=R+2+rK_X^2\geq 3r\geq 6$
    \end{enumerate}
    \item $-D$ is slo divisor. \begin{enumerate}
        \item $r=0$, $H^k(D-rK_X)=H^k(D)=0$ since $h^k(-(-D))=h^k(D)=0$ for all $k\geq 1$. 
        \item $r=1$, $H^k(D-rK_X)=H^k(D-K_X)$. Again, we have long exact sequence of cohomology of line bundles:
        $$ \ldots\rightarrow H^1(D)\rightarrow H^1(D-K_X)\rightarrow H^1(O_C(D-K_X))\rightarrow H^2(D)\rightarrow H^2(D-K_X)\rightarrow 0 $$
        $-D$ is slo, then $-D$ is of the form $A_i+\ldots+A_j$ where $[i,j]$ is (cyclic)interval of length $\leq n$. But $-K_X=A_1+\ldots+A_{n+1}$. Then $D-K_X$ is also of the form $A_l+\ldots+A_s$, where $[l,s]$ is a (cyclic)interval of length $\leq n$. So $D-K_X$ is also a slo divisor since $\mathcal{A}$ is cyclic strong exceptional toric system. Then $H^k(D-K_X)=0$ for $k\geq 1$.
        \item $r\geq 2$. Similarly, we prove by the induction and assume that $H^k(D-(r-1)K_X)=0$ for $k\geq 1$ and $r\geq 2$. Then we also have long exact sequence:
        $$ \ldots\rightarrow H^k(O_X(D-(r-1)K_X)\rightarrow H^k(D-rK_X)\rightarrow H^k(O_C(D-rK_X))\rightarrow\ldots $$ 
        $-D$ is slo divisor, so $(-D)^2+(-D).(K_X)=-2$, thus $D^2-D.K_X=-2$. We have: $(-D).K_X=-2-D^2=-2-R$. $deg_C(O_C(D-rK_X))=(-K_X).D+rK_X^2=rK_X^2-R-2$.
        $-D$ is slo divisor of the form $A_i+\ldots+A_j$, where $[i,j]$ is a (cyclic)interval of length $\leq n$. Then $-2\leq (-D)^2=D^2=R\leq K_X^2-2$ by Theorem $4.7,(5)$. Thus $3\leq (r-1)K^2\leq rK_X^2-R-2\leq rK_X^2$. Therefore, $H^k(O_C(D-rK_X))=0$ for $r\geq 2,k\geq 1$. Then by the induction hypothesis, $H^k(D-rK_X)=0$. 
    \end{enumerate}
\end{enumerate}
Thus, $T$ is a pull back tilting bundle.
\end{proof}

\begin{remark}
In fact, under the assumption that $T$ is a direct sum of an exceptional collection of line bundles, one is also able to show that $T$ is a cyclic strong tilting bundle iff $T$ is of generation time $2$. Check \cite[Proposition 16.4]{EXZ} for more details. 
\end{remark}

Next, we give a complete classifications of surfaces having a $2$-hereditary tilting bundle $T$ which is given by a direct sum of exceptional collection of line bundles. The next proposition is due to D.Chan in \cite{CHAN}  We give a new proof by using criterion in Theorem $4.7$. The toric case is worked out in \cite{HP}, check Theorem $8.6$ in \cite{HP} for more details.

\begin{predl}\cite[Theorem 6.5]{CHAN}
Let $X$ be del Pezzo surface of degree $K_X^2\geq 3$. Then the bundle $T$ below are $2$-hereditary tilting bundles on $X$.
\begin{enumerate}
    \item If $r=1$ then $T=O_X(E_1)\oplus O_X(H)\oplus O_X(H+E_1)\oplus O_X(2H)$
    \item If $r=1$ then $T=O_X(E_{12})\oplus O_X(H+E_1)\oplus O_X(H+E_2)\oplus O_X(H+E_{12})\oplus O_X(2H)$
    \item If $3\leq r\leq 6$ then 
    $$T=O_X(E_{123})\oplus\ldots\oplus O_X(E_{123r})\oplus O_X(H+E_{23})\oplus O_X(H+E_{13})\oplus O_X(H+E_{12})\oplus O_X(H+E_{123})\oplus O_X(2H)$$
    
    \end{enumerate}
    where $E_{ij\ldots l}=E_i+E_j+\ldots+E_l$
\end{predl}

\begin{proof}
For every collection of line bundles $O_X(D_1),\ldots ,O_X(D_n)$, we construct a collection of divisors $A_1,\ldots,A_n$ by $A_1=D_2-D_1,\ldots, A_{n-1}=D_n-D_{n-1}, A_n=-K_X-\sum_{i=1}^{n-1}A_i$. Then the corresponding collection of divisors are given by:\begin{enumerate}
    \item If $r=1$, $H-E_1,E_1,H-E_1,H$
    \item If $r=2$, $H-E_2,E_2-E_1,E_1,H-E_1-E_2,H$
    \item If $r=3$, $H-E_1,E_1-E_2,E_2-E_3,E_3,H-E_1-E_2-E_3, H$
    \item If $r=4$, $E_4,H-E_1-E_4,E_1-E_2,E_2-E_3,E_3,H-E_1-E_2-E_3, H-E_4$
    \item If $r=5$, $E_4,E_5-E_4,H-E_1-E_5,E_1-E_2,E_2-E_3,E_3,H-E_1-E_2-E_3,H-E_4-E_5$
    \item If $r=6$, $E_4,E_5-E_4,E_6-E_5,H-E_1-E_6,E_1-E_2,E_2-E_3,E_3,H-E_1-E_2-E_3,H-E_4-E_5-E_6$
\end{enumerate}

It is easy to verify that these collections of divisors are all toric systems with $A_i^2\geq -2$ on $X$, so by Corollary $4.9$, they are cyclic strong exceptional toric systems and the corresponding collections of line bundles are cyclic strong exceptional collection of line bundles, by Theorem $1.8$, $T$ is a quasi-canonical $2$-hereditary tilting bundle.
\end{proof}

%\begin{predl}[HP]
%If $X$ is a smooth toric surface with nef anti-canonical divisor, then the bundle $T$ below are $2$-hereditary tilting bundles on $X$.
%\begin{table}[h]
%\label{table_cylicstrong}
%\caption{$2$-hereditary tilting bundles on weak del Pezzo toric surfaces}
%\begin{center}
%\begin{tabular}{|p{0.5cm}|p{3cm}|p{13cm}|}
%		 \hline
%			$5b$ & $(-1,-2,0,1,-1)$ & $T=O_X\oplus O_X(H-E_1)\oplus O_X(H)\oplus O_X(2H-E_1-E_2)\oplus O_X(2H-E_1)$\\
%		 \hline
%			$6b$ & $(-1,-2,-1,-1,1,-1)$ & $T=O_X\oplus O_(H-E_{13})\oplus O_X(H-E_3)\oplus O_X(2H-E_{123})\oplus O_X(2H-E_{13})\oplus O_X(3H-E_{123}-E_3)$ \\
%		 \hline
%			$6c$ & $(-1,-2,0,0,-1,-2)$ &  $T=O_X\oplus O_(H-E_{13})\oplus O_X(H-E_3)\oplus O_X(2H-E_{123})\oplus O_X(2H-E_{13})\oplus O_X(3H-E_{123}-E_3)$\\
%		 \hline	
%			$6d$ & $(-1,-2,-2,0,1,-2)$ &  $T=$ \\
%		 \hline	
%			$7a$ & iff $h^0(K_X+D)=0$ & iff $h^0(K_X+D)=0$ \\
%		 \hline	
%		    $7b$ &          & \\
%		 \hline
%		    $8a$ & &\\
%		 \hline
%		    $8b$ & &\\
%		 \hline
%		    $8c$ & &\\
%		 \hline
%		    $9$  & &\\
%		 \hline
%\end{tabular}
%\end{center}
%\end{table}
%\end{predl}

\begin{predl}\cite[Theorem 15.3]{EXZ}
Toric systems in table $6$ are cyclic strong exceptional. Therefore weak del Pezzo surfaces from table $6$ admit cyclic strong exceptional toric systems and the $2$-hereditary tilting bundles are constructed as $T=O_X\oplus O_X(A_1)\oplus\ldots\oplus O_X(A_1+\ldots+A_{n-1})$, where $\mathcal{A}=(A_1,\ldots,A_n)$ are cyclic strong exceptional toric systems in the table. Moreover, those surfaces in table $6$ are the only surfaces admitting such systems.

\begin{table}[h]
\begin{center}
\caption{Cyclic strong exceptional toric systems}
\label{table_yes}
\begin{tabular}{|p{2cm}|p{2cm}|p{10cm}|}
 \hline
 degree & type & toric system \\
 \hline
 $9$ & $\P^2$ & $L,L,L$  \\
 \hline
 $8$ & $F_0$ & $H_1,H_2,H_1,H_2$  \\
 \hline
 $8$ & $F_1$ & $L_1,E_1,L_{1},L$  \\
 \hline
 $8$ & $F_2$ & $F,S-F,F,S-F$ (where $F^2=0, S^2=2, FS=0$)  \\
 \hline
 $7$ & any & $L_{1},E_1,L_{12},E_2,L_{2}$  \\
 \hline
 $6$ & any & $L_{13},E_1,L_{12},E_2,L_{23},E_3$  \\
 \hline
 $5$ &  $\emptyset$ & ${L_{134}},E_4,{E_1-E_4},L_{12},E_2,L_{23},E_3$\\
 \hline
 $5$ &  $A_1$ &  the above \\  \hline
 $5$ &  $2A_1$ & the above \\  \hline
 $5$ &  $A_2$ & the above  \\  \hline
 $5$ &  $A_1+A_2$ & the above \\  \hline
 $4$ & $\emptyset$ & ${L_{134}},E_4,{E_1-E_4},L_{12},{E_2-E_5},E_5,{L_{235}},E_3$  \\
 \hline
 $4$ & $A_1$ & the above  \\  \hline
 $4$ & $2A_1,9$ & the above  \\  \hline
 $4$ & $2A_1,8$ & the above  \\  \hline
 $4$ & $A_2$ & the above  \\  \hline
 $4$ & $3A_1$ & the above  \\  \hline
 $4$ & $A_1+A_2$ & the above  \\  \hline
 $4$ & $A_3,4$ & the above  \\  \hline
 $4$ & $4A_1$ & the above  \\  \hline
 $4$ & $2A_1+A_2$ & the above  \\  \hline
 $4$ & $A_1+A_3$ & the above  \\  \hline
 $4$ & $2A_1+A_3$ & the above  \\  \hline
 $3$ & $\emptyset$ & ${E_2-E_4,L_{125}},E_5,{E_1-E_5,L_{136}},E_6,{E_3-E_6,L_{234}},E_4$   \\  \hline
 $3$ & $A_1$ & the above  \\  \hline
 $3$ & $2A_1$ & the above  \\  \hline
 $3$ & $A_2$ & the above  \\  \hline
 $3$ & $3A_1$ & the above  \\  \hline
 $3$ & $A_1+A_2$ & the above  \\  \hline
 $3$ & $4A_1$ & the above  \\  \hline
 $3$ & $2A_1+A_2$ & the above  \\  \hline
 $3$ & $2A_2$ & the above  \\  \hline
 $3$ & $A_1+2A_2$ & the above  \\  \hline
 $3$ & $3A_2$ & the above  \\  \hline
\end{tabular}
\end{center}
\end{table}

\end{predl}

\begin{proof}
Check Proposition $15.3$ in \cite{EXZ} for details.
\end{proof}

Before stating the next corollary, let us introduce a notion: $C(X):=\bigoplus_{r\geq 0}H^0(X,O(-rK_X))$. It is called the cone of a weak del Pezzo surface $X$. There is a natural map from $Tot(\omega_X)$ to $C(X)$ by contracting the zero section. Note that $Tot(\omega_X)$ has a trivial canonical line bundle and it is a crepant resolution of $C(X)$. 

The classification in Proposition $8.7$ gives a direct application in constructing noncommutative crepant resolutions in the sense of Van den Berg \cite[Definition 4.1]{Van}:

\begin{corollary}
Let $X$ be a weak del Pezzo surface admitting a cyclic strong exceptional collection of line bundles $\mathbb{E}:=(E_0,\ldots,E_n)$. If $T$ is the corresponding tilting bundle. Then:\begin{enumerate}
    \item $D^b(Tot(\omega_X))$ is equivalent to $D^b(A)$, where $A=End(\pi^*(T))$ and $\pi: Tot(\omega_X)\rightarrow X$ is the projection.
    \item $A$ is noncommutative crepant resolution of $C(X)$. 
    
\end{enumerate}
\end{corollary}

\begin{proof}

\begin{enumerate}
    \item $T$ is a tilting bundle formed by cyclic strong exceptional collection of line bundles, by Proposition $8.4$ $\pi^*(T)$ is a tilting bundle on $Tot(\omega_X)$, then the statement follows from Rickardson's result.
    \item It directly follows from the proof of \cite[Proposition 7.2]{Van}.
    \end{enumerate}
\end{proof}

\section{Torsion Exceptional Sheaves on Weak del Pezzo surfaces}
Let $X$ be a smooth projective surface over an algebraically closed field $k$ of characteristic $0$. Let $D^b(X)$ be bounded derived category of coherent sheaves on $X$. We say an object $M\in D^b(X)$ is spherical if we have $M\otimes\omega_X\cong M$ and $Hom(M,M[i])=0$ for $i=1$, $Hom(M,M[i])\cong k$ for $i=0,2$. We consider the mapping cone: 
$$C=Cone(\pi_1^*M^{\vee}\otimes\pi_2^*M\rightarrow O_{\Delta})$$
where $\Delta\subset X\times X$ is the diagonal and $\pi_i$ is the projection from $X\times X$ to the $i$-th factor. Then the integral functor $T_M:=\Phi^C_{X\rightarrow X}$ defined an auto-equivalence of $D^b(X)$, called the twisted functor associated to spherical object $M$. By definition, for $N\in D^b(X)$, we have an exact triangle:
$$\mathbb{R}Hom(M,N)\otimes M\xrightarrow{ev}N\rightarrow T_M(N) $$

We are ready to prove the following statement:
\begin{theorem}
Let $X$ be a weak del Pezzo surface such that $K_X^2\geq 2$, let $\mathcal{E}$ be torsion exceptional sheaf of the form $O_D$, where $D$ is subscheme of codimension $1$, not neccesarily reduced. Then there exists a series of spherical twist functors associated to line bundles on $(-2)$-curves: $T_i\in Auteq(D(X))$ and an $(-1)$-curve $D$ on $X$, such that $T_n\circ T_{n-1}\circ\ldots\circ T_1 (O_D)=O_E$.
\end{theorem}

The next lemma gives a one to one correspondence between torsion exceptional sheaf $O_D$ and $(-1)$-slo divisor $D$: 

\begin{lemma}
Let $D$ be a strong left-orthogonal divisor with $D^2=-1$ on weak del Pezzo surface $X$ with $K_X^2\geq 2$. Let $(O_X(-D),O_X)$ be an exceptional collection of line bundles of length $2$(it is called exceptional pair). Then the right mutation of the exceptional pair $(O_X(-D),O_X)$ gives exceptional pair $(O_X,O_D)$. Conversely, if $\mathcal{E}$ is a torsion exceptional sheaf of the form $O_D$, then $D$ is a slo divisor with $D^2=-1$. 
\end{lemma}

\begin{proof}
$D$ is slo divisor, then $D$ is connected and $h^0(O_X(D))=D^2+2=1$. $Hom_{O_X}(O_D,O_D)\cong Hom_{O_D}(O_D,O_D)\cong\mathbb{C}$. Note that $Ext^i(O_X(-D),O_X)=0,i\geq 1, Hom(O_X(-D),O_X)=\mathbb{C}$ because $D$ is slo with $D^2=-1$. Thus the right mutation of the exceptional pair $(O_X(-D),O_X)$ is just $(O_X,R_{O_X(-D)}(O_X))$. Consider standard exact sequence:
$$0\rightarrow O_X(-D)\xrightarrow{j} O_X\rightarrow O_D\rightarrow 0$$
$j$ is injective, thus $R_{O_X(-D)}(O_X)\cong coker(j)\cong O_D$ by \cite[Section 3.1]{KU}. Thus $(O_X,O_D)$ is an exceptional pair and $O_D$ is exceptional sheaf. $Supp(D)$ is a tree of projective lines since $D$ is an effective lo divisors. $X$ is weak del Pezzo surface, $C^2\geq -2$ for all smooth rational curves $C$. By \cite[Corollary 6.1]{E}, Effective lo divisor $D$ with $D^2=-1$ means that $D$ is a tree of one $(-1)$-curve and $(-2)$-curves(it is possible that there is no $(-2)$-curve). By the classification of exceptional sheaves on surface with nef anticanonical bundle in [Ku], $O_D$ is a torsion exceptional sheaf. Conversely, let $O_D$ be a torsion exceptional sheaf. Thus $c_1(O_D)^2=D^2=-1$ by Hirzebruch-Riemann-Roch theorem. $Supp(O_D)$ is connected and consists of only one $(-1)$-curve and $(-2)$-curves on weak del Pezzo surfaces by \cite[Lemma 2.9, Lemma 5.1]{CJ}. Thus $D.K_X=-1$, hence $D^2+D.K_X=-2$. This means that $D$ is a $(-1)$-class on weak del pezzo surface. By Proposition $3.2$, it is a slo divisor with $D^2=-1$.
\end{proof}

The next lemma belongs to \cite[Lemma 1]{H}
\begin{lemma}
Let $X$ be a smooth projective surface and $C$ is a $(-2)$-curve on $X$. Then the spherical twist $T_{O_C(-1)}$ maps a line bundle $L=O(D)$ to another line bundle $L'$ if and only if one of the following two cases apply\begin{enumerate}
   \item $D.C=0$, in which case $L$ and $L'$ are equal, or 
   \item $D.C=1$, in which case $L'=L\otimes O(C)=O(D+C)$.
\end{enumerate}
\end{lemma}

\begin{proof}(Proof of Theorem $9.1$ or $1.12$)
Let $O_D$ be a torsion exceptional sheaf on $X$. then by Lemma $9.2$, $D$ is a slo divisor with $D^2=-1$. By \cite[Corollary 6.1]{E}, $D=E+\sum_{i=1}^nk_iC_i$, where $E$ is a $(-1)$-curve and $C_i$ is a $(-2)$-curve and the underlying dual graph of $D$ is a tree. Note that for any $(-2)$-curve $C_i\subset D$, $-1\leq C_i.D\leq 1$ and if $D.C_i\geq 0$ for all $(-2)$-curves $C_i$, then $D$ is already a $(-1)$-curve by \cite{Do}. 
%(Observe that intersection number of $(-1)$-class and $(-2)$-class on rational surface of degree greater than $1$ is $-1,0,1$). 
If $C_1.D=-1$, let $D_1=D-C_1$. Then $D_1^2=-1$ and by Proposition $3.4$, $D_1$ is a $(-1)$-slo divisor. $O_{C_1}(-1)$ is a spherical object in $D^b(X)$. Consider standard exact sequence:
$$0\rightarrow O_X(-D)\rightarrow O_X\rightarrow O_D\rightarrow 0 $$
Apply the spherical twist functor $T_1=T_{O_{C_1}(-1)}$ to the exact sequence and note that $(-D).C_1=1$. By Lemma $9.3$, we get an exact sequence:
$$0\rightarrow T_{O_{C_1}(-1)}(O(-D))\rightarrow T_{O_{C_1}(-1)}(O_X)\rightarrow T_{O_{C_1}(-1)}(O_D)\rightarrow 0 $$
i.e:
$$ 0\rightarrow O(-D_1)\rightarrow O_X\rightarrow T_{O_{C_1}(-1)}(O_D)\rightarrow 0 $$

Thus, we have $T_{O_{C_1}(-1)}(O_D)\cong O_{D_1}$. 

If $C_2.D_1=-1$, again, let $D_2=D_1-C_2$. Then $D_2$ is still a $(-1)$-slo divisor and $(-D_1).C_2=1$. Applying the spherical functor $T_{O_{C_2}(-1)}$ to the exact sequence above, we get:
$$ 0\rightarrow O(-D_2)\rightarrow O_X\rightarrow T_{O_{C_2}(-1)}\circ T_{O_{C_1}(-1)}(O_D)\rightarrow 0 $$

Thus $T_{O_{C_2}(-1)}\circ T_{O_{C_1}(-1)}(O_D)\cong O_{D_2}$ 

Continuing  in this way, it will terminate after finitely many steps since there are finitely many $(-1)$-classes on $X$. %Further, if for a $(-1)$-class $E$, $E.C_i\geq 0$ for all $(-2)$-curves, then $E$ is a $(-1)$-curve by [Do]. 
Eventually, we get the following exact sequence:
$$ 0\rightarrow O(-E)\rightarrow O_X\rightarrow T_{O_{C_n}(-1)}\circ\ldots\circ T_{O_{C_2}(-1)}\circ T_{O_{C_1}(-1)}(O_D)\rightarrow 0 $$

Hence, $T_{O_{C_n}(-1)}\circ\ldots\circ T_{O_{C_2}(-1)}\circ T_{O_{C_1}(-1)}(O_D)\cong O_E$, 
where $E$ is an irreducible $(-1)$-curve. 
\end{proof}

%\begin{thebibliography}{5}
%\bibitem{height}
%Alexander Kuznetsov
%\textit{Height of exceptional collections and Hochschild cohomology of quasiphantom categories}
%arXiv preprint arXiv:1404.3143 (2014).

%\bibitem{criterion}
%Alexey Elagin
%\textit{A criterion for left-orthogonality of an effective divisor on a surface}
%arXiv preprint arXiv:1610.02325 (2016).

%\bibitem{BP} Brown, Morgan, and Ian Shipman 
%\textit{The McKay correspondence, tilting equivalences, and rationality}
%arXiv preprint arXiv:1312.3918 (2013).
%\bibitem{K1} A.\,King, ``Tilting bundles on some rational surfaces'', unpublished manuscript, see \textsl{http://www.maths.bath.ac.uk/~masadk/papers/}, 1997.

%\bibitem{AL} Valery Lunts, Alexey Elagin {\it On full exceptional collections of line bundles on del Pezzo surfaces}.

%\end{thebibliography}

\end{document}